\documentclass{article}

\RequirePackage[OT1]{fontenc}
\RequirePackage[%lnms,
amsthm,amsmath%,natbib
]{imsart}

\usepackage{graphicx}
\usepackage{latexsym,amsmath}
\usepackage{amsmath,amsthm,amscd}
\usepackage{amsfonts}
\usepackage[psamsfonts]{amssymb}
\usepackage{enumerate}

\usepackage{url}

\usepackage{tocvsec2}

\usepackage{epstopdf}
\usepackage{float}

\startlocaldefs
\numberwithin{equation}{section}

\theoremstyle{plain} 
\newtheorem{theorem}{Theorem}[section]
\newtheorem{corollary}[theorem]{Corollary}

\newtheorem{proposition}[theorem]{Proposition}

\theoremstyle{definition} 

\theoremstyle{definition} 

\newtheorem*{ex*}{Example}
\theoremstyle{remark} 

\theoremstyle{remark} 
\newtheorem{remark}[theorem]{Remark}
\newtheorem*{remark*}{Remark}
\numberwithin{equation}{section}

\newcommand{\dd}{\partial}

\renewcommand{\dd}{{\operatorname{d}}}

\newcommand{\eqD}{\overset{\operatorname{D}}=}

\newcommand{\ch}{\operatorname{ch}}
\newcommand{\sh}{\operatorname{sh}}

\newcommand{\si}{\sigma}

\newcommand{\ka}{\kappa}
\newcommand{\la}{\lambda}

\newcommand{\ii}[1]{\,\mathbf{I}\{#1\}} 

\newcommand{\pd}[2]{\operatorname{\partial}_{#2}{#1}} 

\newcommand{\intr}[2]{\overline{#1,#2}}

\renewcommand{\P}{\operatorname{\mathsf{P}}} 
\newcommand{\E}{\operatorname{\mathsf{E}}}
\newcommand{\Var}{\operatorname{\mathsf{Var}}}

\newcommand{\R}{\mathbb{R}}

\newcommand{\EE}{\mathcal{E}}

\newcommand{\tc}{{\tilde{c}}}

\newcommand{\trho}{{\tilde{\rho}}}

\newcommand{\tZ}{{\tilde{Z}}}

\renewcommand{\le}{\leqslant}
\renewcommand{\ge}{\geqslant}

\newcommand{\X}{{\mathfrak{X}}}
\newcommand{\W}{{\mathfrak{P}}}

\endlocaldefs

\begin{document}

\begin{frontmatter}

\title{Optimal re-centering bounds, with applications to Rosenthal-type concentration of measure inequalities}
\runtitle{Re-centering bounds}
%\date{\today}

% \author{\fnms{First}  \snm{Author}\corref{}\thanksref{t2}\ead[label=e1]{first@somewhere.com}},
%  \author{\fnms{Second} \snm{Author}\ead[label=e2]{second@somewhere.com}}
%  \and
%  \author{\fnms{Third}  \snm{Author}%
%  \ead[label=e3]{third@somewhere.com}%
%  \ead[label=u1,url]{http://www.foo.com}}
%
%  \thankstext{t2}{Footnote to the first author with the `thankstext' command.}

\begin{aug}
\author{\fnms{Iosif} \snm{Pinelis}%\thanksref{t2}\ead[label=e1]{ipinelis@mtu.edu}
}
%  \thankstext{t2}{Supported by NSF grant DMS-0805946}
\runauthor{Iosif Pinelis}

%\affiliation{Michigan Technological University}

\address{Department of Mathematical Sciences\\
Michigan Technological University\\
Houghton, Michigan 49931, USA\\
E-mail: \printead[ipinelis@mtu.edu]{e1}}
\end{aug}

\begin{abstract} 
For any nonnegative Borel-measurable function $f$ such that $f(x)=0$ if and only if $x=0$, the best constant $c_f$ in the inequality $\E f(X-\E X)\le c_f\E f(X)$ for all random variables $X$ with a finite mean is obtained. Properties of the constant $c_f$ in the case when $f=|\cdot|^p$ for $p>0$ are studied. 
Applications to concentration of measure in the form of Rosenthal-type bounds on the moments of separately Lipschitz functions on product spaces are given. 
\end{abstract}

%\subjclass[2000]{60E15, 62G10, 62G15, 60G50, 62G35}
% 62G10    	Hypothesis testing
%  62G15    	Tolerance and confidence regions
%  60G50    	Sums of independent random variables; random walks
%   62G35    	Robustness
  
%
%\keywords{probability inequalities; Rade\-macher random variables; sums of independent random variables; Student's test; self-normalized sums}


% 60B11    	Probability theory on linear topological spaces
% 46B09    	Probabilistic methods in Banach space theory [See also 60Bxx]
% 46B20    	Geometry and structure of normed linear spaces
% 46B10    	Duality and reflexivity [See also 46A25]


\end{frontmatter}

\settocdepth{chapter}

\tableofcontents 
%%%%%%%%%%%%%%%%%{\small\tableofcontents} 

\settocdepth{subsubsection}

\theoremstyle{plain} 
\numberwithin{equation}{section}

%\eject

\section{Introduction}\label{intro} 

In many situations (as e.g.\ in \cite{nonlinear}), one starts with zero-mean random variables (r.v.'s), which need to be truncated in some manner, and then the means no longer have to be zero. So, to utilize such tools as the Rosenthal inequality for sums of independent zero-mean r.v.'s, one has to re-center the truncated r.v.'s. Then one will usually need to bound moments of the re-centered truncated r.v.'s in terms of the corresponding moments of the original r.v.'s. 
To be more specific, let $Z$ be a given r.v., possibly (but not necessarily) of zero mean. Next, let $\tZ$ be a truncated version of $Z$ such that $|\tZ|\le|Z|$; possibilities here include letting $\tZ$ equal $Z\ii{Z\le z}$ or $Z\ii{|Z|\le z}$ or $Z\wedge z$, for some $z>0$; cf.\ \cite{winzor,tilted-mean}. 
Assume that $\E|\tZ|<\infty$. 
Then for any $p\ge1$ one can 
use the inequalities $|x-y|^p\le2^{p-1}(|x|^p+|y|^p)$ and $(\E|\tZ|)^p\le\E|\tZ|^p$,  	
to 
write 
\begin{equation}\label{eq:Z}
\E|\tZ-\E\tZ|^p%\le2^{p-1}(\E|\tZ|^p+|\E\tZ|^p)
\le2^p\E|\tZ|^p\le2^p\E|Z|^p,  	
\end{equation}
as is oftentimes done. 
However, the factor $2^p$ in \eqref{eq:Z} can be significantly improved, especially for $p\ge2$. For instance, it is clear that for $p=2$ this factor can be reduced from $2^2=4$ to $1$. 
More generally, for every real $p>1$ we shall provide the best constant factor $C_p$ in the inequality 
\begin{equation}\label{eq:p}
	\E|X-\E X|^p\le C_p\E|X|^p 
\end{equation}
for all r.v.'s $X$ with a finite mean $\E X$. 
%We shall also show that, 
In particular, 
$C_p$ improves the factor $2^p$ more than $6$ times for $p=3$, and  
for large $p$ this improvement is asymptotically $\sqrt{8ep}$ times; see parts \eqref{C_3} and \eqref{asymp} of Theorem~\ref{prop:p} and the left panel in Figure~\ref{fig:graphs} in this paper. 
In fact, in Theorem~\ref{prop:centring} below we shall present an extended version of the exact inequality \eqref{eq:p}, for a quite general class of moment functions $f$ in place of the power functions $|\cdot|^p$. 

Another natural application of these results is to concentration of measure for separately Lipschitz functions on product spaces. In Section~\ref{concentr} of this paper, we shall give 
Rosenthal-type bounds on the moments of such functions. Similar extensions of the von Bahr--Esseen inequality were given in \cite{bahr-esseen}. 

\section{Summary and discussion}\label{summary}  

%\subsection{Summary}\label{summary} 
Let $f\colon\R\to\R$ be any nonnegative Borel-measurable function such that $f(x)=0$ if and only if $x=0$. 
Let $X$ stand for any %nondegenerate 
random variable (r.v.) with a finite mean $\E X$. %, so that $\P(X=\E X)<1$. 

\begin{theorem}\label{prop:centring}
One has
\begin{equation}\label{eq:centring}
	\E f(X-\E X)\le c_f\E f(X), 
\end{equation}
where 
\begin{equation}\label{eq:cf}
	c_f:=\sup\Big\{\frac{af(b)+bf(-a)}{af(b-t)+bf(-a-t)}\colon a\in(0,\infty), b\in(0,\infty), % a+b>0,  
t\in\R\Big\} 
\end{equation}
is the best possible constant factor in \eqref{eq:centring} (over all r.v.'s $X$ with a finite mean). 
\end{theorem} 

All necessary proofs will be given in Section~\ref{proofs}. 

Note that for all $a\in(0,\infty)$, $b\in(0,\infty)$, and  
$t\in\R$ both the numerator and the denominator of the ratio in \eqref{eq:cf} are strictly positive (since $f$ is nonnegative and vanishes only at $0$). 
So, $c_f$ is correctly defined, with possible values in $(0,\infty]$. 

It is possible to say much more about the optimal constant factor $c_f$ in the important case when $f$ is the power function $|\cdot|^p$. To state the corresponding result, let us introduce more notation. 

Take any $a\in(0,\infty)$ and $b\in(0,\infty)$, and let $X_{a,b}$ be any zero-mean r.v.\ with values $-a$ and $b$, so that 
\begin{equation*}
	\P(X_{a,b}=b)=\frac a{a+b}=1-\P(X_{a,b}=-a). 
\end{equation*}
Note that 
\begin{equation*}%\label{eq:eqD}
	X_{b,a}\eqD-X_{a,b}, 
\end{equation*}
where $\eqD$ denotes the equality in distribution. 

Take any 
\begin{equation}\label{eq:p in}
p\in(1,\infty) 	
\end{equation}
and introduce  
\begin{equation}\label{eq:R}
		R(p,b):=(b^{p - 1} + (1 - b)^{p - 1}) \big(b^{\frac1{p - 1}} + (1 - b)^{\frac1{p - 1}}\big)^{p - 1} \quad\text{for any $b\in[0,1]$.}
\end{equation}

\begin{proposition}\label{lem:}
If $p\ne2$ then there exists %a (necessarily unique) number 
$b_p\in(0,\frac12)$ such that 
\begin{enumerate}[(i)]
\item
$\pd{R(p,b)}b>0$ for $b\in(0,b_p)$ and hence 
$R(p,b)$ is (strictly) increasing in $b\in[0,b_p]$; 
\item $\pd{R(p,b)}b<0$ for $b\in(b_p,\frac12)$ and hence $R(p,b)$ is 
decreasing in $b\in[b_p,\frac12]$. 
\end{enumerate}
So, $b_p$ is the unique maximizer of $R(p,b)$ over all $b\in[0,\frac12]$. 
\end{proposition}

In Proposition~\ref{lem:} and in the sequel, $\pd{}\cdot$ denotes the partial differentiation with respect to the argument in the subscript. 

\begin{theorem}\label{prop:p}\ 
\begin{enumerate}[(i)]
	\item\label{ineq} Inequality \eqref{eq:p} holds with the constant factor 
\begin{gather}
	C_p:=c_{|\cdot|^p}=\sup_{b\in[0,1]}R(p,b)=\max_{b\in(0,1/2)}R(p,b)=R(p,b_p),   
	\label{eq:C_p} 
\end{gather}	
where $R(p,b)$ is as in \eqref{eq:R}  
and $b_p$ is as in Proposition~\ref{lem:}. 
In particular, $C_2=R(2,b)=1$ for all $b\in[0,1]$. 
\item \label{best}
$C_p$ is the best possible constant factor in \eqref{eq:p}. More specifically, 
the equality in \eqref{eq:p} obtains if and only if one of the following three conditions holds: 
\begin{enumerate}[(a)]
	\item $\E|X|^p=\infty$; 
	\item $p=2$, $\E X^2<\infty$, and $\E X=0$; 
		\item $p\ne2$ and $X\eqD\la(X_{1-b_p,b_p}-t_{b_p})$ for some $\la\in\R$, where 
\begin{equation}\label{eq:t_b}
	t_b:=b-\frac{b^{1/(p-1)}}{b^{1/(p-1)}+(1-b)^{1/(p-1)}} 
\end{equation}
for all $b\in(0,1)$, and $b_p$ is as in Proposition~\ref{lem:}.
\end{enumerate} 
\item \label{symm}
One has the symmetries 
\begin{equation}\label{eq:symm}
	C_p^{1/\sqrt{p-1}}=C_q^{1/\sqrt{q-1}}\quad\text{and}\quad b_p=b_q, 
\end{equation}
where $q$ is dual to $p$ in the sense of $L^p$-spaces: 
\begin{equation*}
	\frac1p+\frac1q=1. 
\end{equation*}
\item \label{asymp}
For $p\to\infty$, 
\begin{equation}\label{eq:C_p sim}
	C_p\sim\frac{2^p}{\sqrt{8ep}}; 
\end{equation}
as usual, $A\sim B$ means that $A/B\to1$. 
\item \label{C_p mono}
$C_p$ is strictly log-convex and hence continuous in $p\in(1,\infty)$; moreover, $C_p$ decreases in $p\in(1,2]$ from $2$ to $1$ and increases in $p\in[2,\infty)$ from $1$ to $\infty$. 
\item \label{C_3} The values of $C_p$, $b_p$, and $t_{b_p}$ are algebraic whenever $p$ is rational; in particular, %C3,b3,t3.nb 
$C_3=\frac1{27} (17+7 \sqrt 7)=1.315...$, $b_3=\frac12-\frac16\,\sqrt{1+2 \sqrt{7}}
%=\frac{1}{6} \left(3-\sqrt{1+2 \sqrt{7}}\right)
=0.0819...$, and $t_{b_3}=-\frac{1}{3} \sqrt{\frac{1}{2} \left(13 \sqrt{7}-34\right)}=-0.148...$. 
\end{enumerate}
\end{theorem} 

By parts \eqref{C_3} and \eqref{C_p mono} of Theorem~\ref{prop:p}, $C_p$ can in principle be however closely bracketed for any real $p\in(1,\infty)$. However, such a calculation may in many cases be inefficient. On the other hand, %part (ii) of Theorem~\ref{prop:p} 
Proposition~\ref{lem:} allows one to bracket the maximizer $b_p$ of $R(b,p)$ however closely and thus, perhaps more efficiently, compute $C_p$ with any degree of accuracy. 

(A part of) the graph of $C_p$ is shown in Figure~\ref{fig:Cp-graph}, and those of $2^p/C_p$ and $b_p$ are shown in Figure~\ref{fig:graphs}. 

\begin{figure}[H]%[htbp]
	\centering
		\includegraphics[width=.8\textwidth]
		{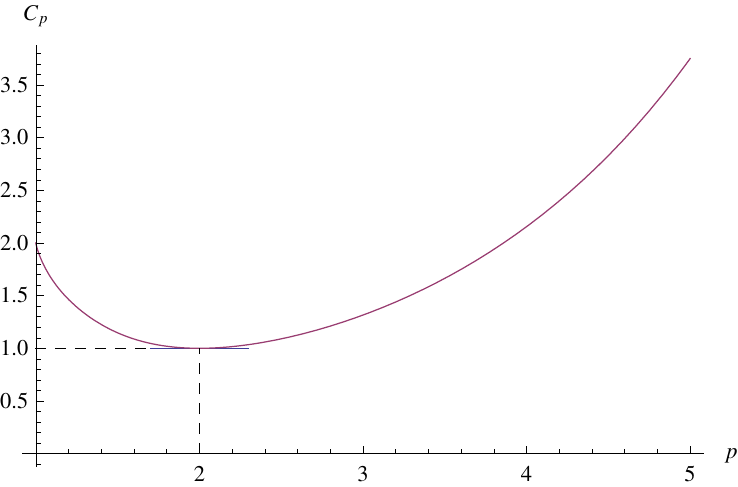}
	\caption{$C_p$ decreases in $p\in(1,2]$ from $2$ to $1$ and increases in $p\in[2,\infty)$ from $1$ to $\infty$. }
	\label{fig:Cp-graph}
\end{figure}

\begin{figure}[H]%[htbp]
	\centering
		\includegraphics[width=1.00\textwidth]{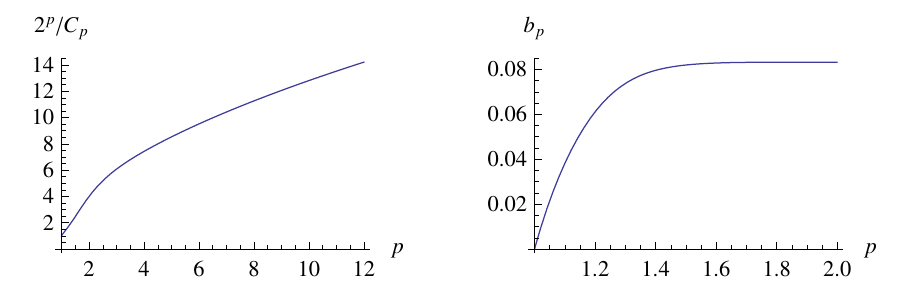}
	\caption{By \eqref{eq:C_p sim}, $2^p/C_p\sim\sqrt{8ep}$ as $p\to\infty$. By \eqref{eq:symm}, $b_p=b_q$; note here also that \break 
$p\in(1,2]\iff q\in[2,\infty)$; by \eqref{eq:b_p sim}, $b_p\sim(p-1)/2$ as $p\downarrow1$.}
	\label{fig:graphs}
\end{figure}

\begin{remark}
What if, instead of the condition \eqref{eq:p in}, one has $p\in(0,1]$? It is easy to see that the inequality \eqref{eq:p} holds for $p=1$ with $C_1=2$ (cf.\ \eqref{eq:Z}), which is then the best possible factor, as seen by letting 
\begin{equation}\label{eq:p=1}
	\text{$X=X_{1-b,b}-b$ with $b\downarrow0$.}
\end{equation}
However, the equality $\E|X-\E X|=2\E|X|$ obtains only if $X\eqD0$; one may also note here that, by part \eqref{C_p mono} of Theorem~\ref{prop:centring}, $C_{1+}=2=C_1$. 
As to $p\in(0,1)$, for each such value of $p$ the best possible factor $C_p$ in \eqref{eq:p} is $\infty$; indeed, consider $X$ as in \eqref{eq:p=1}. %=X_{1-b,b}-b$ with $b\downarrow0$. 
\end{remark}

\section{Application: Rosenthal-type concentration inequalities for separately Lipschitz functions on product spaces}\label{concentr} 

It is well known that for every $p\in[2,\infty)$ there exist finite positive constants $c_1(p)$ and $c_2(p)$, depending only on $p$, such that for any independent real-valued zero-mean r.v.'s $X_1,\dots,X_n$ 
\begin{equation*}%\label{eq:rosenthal}
	\E|Y|^p\le c_1(p)A_p+c_2(p)B^p,
\end{equation*}
where $Y:=X_1+\dots+X_n$, $A_p:=\E|X_1|^p+\dots+\E|X_n|^p$, and $B:=(\E X_1^2+\dots+\E X_n^2)^{1/2}$. 
An inequality of this form was first proved by Rosenthal \cite{rosenthal}, and has since been very useful in many applications. 
It was generalized to martingales \cite[(21.5)]{burk}, including martingales in Hilbert spaces \cite{pin80} and, further, in $2$-smooth Banach spaces \cite{pin94}. 
The constant factors $c_1(p)$ and $c_2(p)$ were actually allowed in \cite{pin80} and \cite{pin94} to depend on certain freely chosen parameters, which provided for optimal in a certain sense sizes of $c_1(p)$ and $c_2(p)$, for any given positive value of the Lyapunov ratio $A_p/B^p$. Best possible Rosenthal-type bounds for sums of independent real-valued zero-mean r.v.'s were given, under different conditions, by Utev \cite{utev-extr} and Ibragimov and Sharakhmetov \cite{ibr-shar97,ibr-sankhya}.  
Also for sums of independent real-valued zero-mean r.v.'s $X_1,\dots,X_n$, Lata{\l}a \cite{latala-moments} obtained an expression $\EE$ in terms of $p$ and the individual distributions of the $X_i$'s such that $a_1\EE\le\|Y\|_p\le a_2\EE$ for some positive absolute constants $a_1$ and $a_2$. 

Given a Rosenthal-type upper bound for real-valued martingales, one 
can use the Yurinski{\u\i} martingale decomposition \cite{yurinskii} and (say) Theorem~\ref{prop:p} to obtain a corresponding upper bound on the $p$th absolute \emph{central} moment of the norm of the sum of independent random vectors in an arbitrary separable Banach space; even more generally, one can obtain such a measure-concentration inequality for separately Lipschitz functions on product spaces.  

To state such a result, let $X_1,\dots,X_n$ be independent r.v.'s with values in measurable spaces $\X_1,\dots,\X_n$, respectively. Let $g\colon\W\to\R$ be a measurable function on the product space $\W:=\X_1\times\dots\times\X_n$. 
Let us say (cf.\ \cite{bent-isr,normal}) that $g$ is {\em separately Lipschitz} if it satisfies a Lipschitz-type condition in each of its arguments: 
\begin{equation}\label{eq:Lip}
|g(x_1,\dots,x_{i-1},\tilde x_i,x_{i+1},\dots,x_n) -
g(x_1,\dots,x_n)| \le \rho_i(\tilde x_i,x_i)
\end{equation}
for some measurable functions $\rho_i\colon\X_i\times\X_i\to\R$ and 
all $i\in\intr1n$, $(x_1,\dots,x_n)\in\W$, and $\tilde x_i\in\X_i$. 
Take now any separately Lipschitz function $g$ and let  
$$Y:=g(X_1,\dots,X_n).$$
Suppose that the r.v.\ $Y$ has a finite mean. 

On the other hand, take any $p\in[2,\infty)$ and suppose that 
positive constants $c_1(p)$ and $c_2(p)$ are such that for all 
real-valued martingales $(\zeta_j)_{j=0}^n$ with $\zeta_0=0$ and differences $\xi_i:=\zeta_i-\zeta_{i-1}$ 
%for $i\in\intr2n$ 
%one has 
\begin{equation}\label{eq:mart}
	\E|\zeta_n|^p\le c_1(p)\sum_1^n\E|\xi_i|^p+c_2(p)\Big(\sum_1^n\|\E_{i-1}\xi_i^2\|_\infty\Big)^{p/2}, 
\end{equation}
where $\E_j$ denotes the expectation given $\zeta_0,\dots,\zeta_j$. 

Then 
one has 

\begin{corollary}\label{cor:concentr}
For each $i\in\intr1n$, take any $x_i$ and $y_i$ in $\X_i$.  
Then 
\begin{equation}\label{eq:concentr}
	\E|Y-\E Y|^p
	\le C_p c_1(p)\sum_1^n\E\rho_i(X_i,x_i)^p+c_2(p)\Big(\sum_1^n\E\rho_i(X_i,y_i)^2\Big)^{p/2},   
\end{equation}
where $C_p$ is as in \eqref{eq:C_p}. 
\end{corollary}

An example of separately Lipschitz functions $g:\X^n\to\R$ is given by the formula 
\begin{equation}\label{eq:g=sum}
	g(x_1,\dots,x_n)=\|x_1+\dots+x_n\|
\end{equation}
for all $x_1,\dots,x_n$ in a separable Banach space $(\X,\|\cdot\|)$. 
In this case, one may take $\rho_i(\tilde x_i,x_i)\equiv\|\tilde x_i-x_i\|$. 
Thus, one immediately obtains 

\begin{corollary}\label{cor:conc-sums}
Let $X_1,\dots,X_n$ be independent random vectors in a Banach space $(\X,\|\cdot\|)$. 
Let here $Y:=\|X_1+\dots+X_n\|$. 
For each $i\in\intr1n$, take any $x_i$ and $y_i$ in $\X_i$.  
Then 
\begin{equation}\label{eq:sum}
	\E|Y-\E Y|^p
	\le C_p c_1(p)\sum_1^n\E\|X_i-x_i\|^p+c_2(p)\Big(\sum_1^n\E\|X_i-y_i\|^2\Big)^{p/2}.    
\end{equation}
\end{corollary}

Particular cases of separately Lipschitz functions more general than the norm of the sum as in \eqref{eq:g=sum} were discussed earlier in \cite{ineqs-largedev11} and \cite[pages~20--23]{viniti10}. 

For $p=2$, it is obvious that the inequality \eqref{eq:mart} holds with $c_1(2)=1$ and $c_2(2)=0$, and then the inequalities \eqref{eq:concentr}  and \eqref{eq:sum} do so. 
Thus, for $p=2$ \eqref{eq:sum} becomes 
\begin{equation}\label{eq:concentr,p=2}
	\Var Y\le\sum_1^n\E\|X_i-x_i\|^2, 
\end{equation}
since $C_2=1$. 
The 
inequality \eqref{eq:concentr,p=2} was presented in \cite[page~29]{viniti10} and \cite[Theorem~4]{pin-sakh}, based on an improvement of the method of Yurinski\u\i\ \cite{yurinskii}; cf.\ \cite{mcdiarmid89,mcdiarmid98,bent-isr}, \cite[Section~4]{normal}, and \cite[Proposition~2.5]{pin94}. The proof of Corollary~\ref{cor:concentr} 
%and \ref{cor:conc-sums} 
is based in part on the same kind of improvement. 

The case $p=3$ is also of particular importance in applications, especially to Berry--Esseen-type bounds; cf.\ e.g.\ \cite[Lemma~A1]{bolt93}, \cite[Lemma~6.3]{chen-shao05}, and \cite{nonlinear}. It follows from the main result of \cite{pin80} that \eqref{eq:mart} holds for $p=3$ with 
$c_1(3)=1$ and $c_2(3)=3$, whereas, by part~\eqref{C_3} of Theorem~\ref{prop:p}, $C_3<1.316$. 
Thus, one has an instance of \eqref{eq:sum} with rather small constant factors:  
\begin{equation*}%\label{eq:sum,p=3}
	\E|Y-\E Y|^3
	\le1.316\,\sum_1^n\E\|X_i-x_i\|^3+3\Big(\sum_1^n\E\|X_i-y_i\|^2\Big)^{3/2}.      
\end{equation*}
Similarly, the more general inequality \eqref{eq:concentr} holds for $p=3$ with $1.316$ and $3$ in place of $C_p c_1(p)$ and $c_2(p)$. 
% cf.: E|Z|^3=1.59577 for Z\sim N(0,1)

As can be seen from the proof given in Section~\ref{proofs}, both Corollaries~\ref{cor:concentr} and \ref{cor:conc-sums} will hold even if 
the separately-Lipschitz condition \eqref{eq:Lip} 
is relaxed to 
\begin{equation}\label{eq:LipE}
	|\E g(x_1,\dots,x_{i-1},\tilde x_i,X_{i+1},\dots,X_n) 
%\\ 
- 
\E g(x_1,\dots,x_i,X_{i+1},\dots,X_n)|\le \rho_i(\tilde x_i,x_i). 
\end{equation}

Note also that in Corollaries~\ref{cor:concentr} and \ref{cor:conc-sums} the r.v.'s $X_i$ do not have to be zero-mean, or even to have any definable mean; at that, the arbitrarily chosen $x_i$'s  and $y_i$'s may act as the centers, in some sense, of the distributions of the corresponding $X_i$'s.  

Other inequalities for the distributions of separately Lipschitz functions on product spaces were given in \cite{bent-isr,normal,bahr-esseen}.  

Clearly, the separate-Lipschitz (sep-Lip) condition \eqref{eq:Lip} is easier to check than a joint-Lipschitz one. Also, sep-Lip (especially in the relaxed form \eqref{eq:LipE}) is more generally applicable. On the other hand, when a joint-Lipschitz condition is satisfied, one can generally obtain better bounds. Literature on the concentration of measure phenomenon, almost all of it for joint-Lipschitz settings, is vast; let us mention here only \cite{ledoux-tala,ledoux_book,lat-olesz,bouch-etal,ledoux-olesz}.  

\section{Proofs}\label{proofs}

\begin{proof}[Proof of Theorem~\ref{prop:centring}]
It is well known that any zero-mean probability distribution on $\R$ is a mixture of zero-mean distributions on sets of at most two elements; see e.g.\ 
\cite[Proposition 3.18]{disintegr}. 
So, % and in view of the condition $f(0)=0$, % and observation, 
there exists a Borel probability measure $\mu$ on the set 
\begin{equation*}
S:=\R\times(0,1/2] 	
\end{equation*}
such that 
\begin{equation}\label{eq:g}
	\E g(X-\E X)=\int_S\E g(\la X_{1-b,b})\,\mu(\dd\la\times\dd b)
%	=\int_{S_0}\E f(\la X_{1-b,b})\,\mu(\dd\la\times\dd b) 
\end{equation}
for all nonnegative Borel functions $g$; the measure $\mu$ depends on the distribution of the r.v.\ $X-\E X$. 
Letting now  
\begin{equation}\label{eq:S_0}
	S_0:=(\R\setminus\{0\})\times(0,1/2] 
\end{equation}
and using the condition $f(0)=0$, one has 
\begin{align}
	\E f(X-\E X)&=\int_S\E f(\la X_{1-b,b})\,\mu(\dd\la\times\dd b) \notag \\ 
	&=\int_{S_0}\E f(\la X_{1-b,b})\,\mu(\dd\la\times\dd b) \notag \\ 
	&\le \tc_f \int_{S_0}\E f(\la X_{1-b,b}+\E X)\,\mu(\dd\la\times\dd b) \label{eq:le tc_f} \\ 
	&\le \tc_f \int_S\E f(\la X_{1-b,b}+\E X)\,\mu(\dd\la\times\dd b) \label{eq:le int S} \\ 
	&=\tc_f \E f\big((X-\E X)+\E X\big)=\tc_f \E f(X), \notag 
\end{align}
where 
\begin{align}
	\tc_f&:=\sup\{\trho_f(\la,b,t)\colon(\la,b)\in S_0, t\in\R\} \quad\text{and} \label{eq:tc_f} \\ 
	\trho_f(\la,b,t)&:=\frac{\E f(\la X_{1-b,b})}{\E f\big(\la(X_{1-b,b}-t)\big)}, \label{eq:trho} 
\end{align}
so that 
\begin{equation}\label{eq:tc_f=c_f}
	\tc_f=c_f. 
\end{equation}
Now the inequality in \eqref{eq:centring} follows from the above multi-line display and \eqref{eq:tc_f=c_f}, and \eqref{eq:tc_f=c_f} \big(together with \eqref{eq:tc_f} and \eqref{eq:trho}\big) also shows that $c_f$ is the best possible constant factor in \eqref{eq:centring}. 
\end{proof}

\begin{proof}[Proof of Proposition~\ref{lem:}] 
%
%\noindent\textbf{(ii)}\quad 
It is straightforward to check 
the symmetry 
\begin{equation}\label{eq:R-symm}
	R(p,b)^{1/\sqrt{p-1}}=R(q,b)^{1/\sqrt{q-1}}  
\end{equation}
for all $b\in[0,1]$, where $q$ is dual to $p$. 

So, it remains to consider  $p\in(1,2)$. 
Also assume that $b\in(0,1/2)$ and introduce %new variables 
\begin{equation}\label{eq:r,x,z}
	r:=p-1,\quad x:=\frac b{1-b},\quad\text{and}\quad z:=-\frac{\ln x}r, 
\end{equation}
so that 
\begin{equation*}
	\text{$r\in(0,1)$,\quad $x\in(0,1)$,\quad and\quad $z\in(0,\infty)$. }
\end{equation*} 
Now introduce  
\begin{align}
	D_1(x)&:=D_1(r,x):=(1-b)\frac{x^r+1}{x^{r-1}-1}\,\pd{\ln R(p,b)}b
	=r - \frac{(x-x^{1/r}) (1 + x^r)}{(x^r-x) (1 + x^{1/r})} \label{eq:D1} \\ 
	\intertext{and} 
	D_2(x)&:=D_2(r,x):=r x^3 (1 + x^{1/r})^2 (x^{r-1}-1)^2\,D_1'(x), 	\label{eq:D2}
\end{align}
so that $D_1(x)$ and $D_2(x)$ equal in sign to $\pd{\ln R(p,b)}b$ and $D_1'(x)$, respectively. 
One can verify the identity 
\begin{equation}\label{eq:D2=}
	D_2(x)e^{(1 + r + r^2)z}/2=D_{21}(z) + (1 - r)D_{22}(z), 
\end{equation}
where 
\begin{align*}
	D_{21}(z)&:=r^2 \sh((1 - r) z) + \sh(r(1 - r) z) - r \sh((1 - r^2) z), \\ 
	D_{22}(z)&:=h(z)-h(rz), \quad h(u):=\sh ru-r\sh u;  
%	\frac{\sh rz - \sh(r^2 z)}{\sh z - \sh rz}-r;   
\end{align*}
we use $\sh$ and $\ch$ for $\sinh$ and $\cosh$. 
Note that $h'(u)=r(\ch ru-\ch u)<0$ for $u>0$ and hence 
\begin{equation*}
	D_{22}(z)<0. 
\end{equation*}
Next, 
\begin{equation*}
	\frac{D_{21}'(z)}{(1 - r) r }
	=
 \big(\ch[(1 - r) r z] - \ch[(1 - r^2) z]\big)
 +r \big(\ch[(1 - r) z] - \ch[(1 - r^2) z]\big) <0, 
\end{equation*}
since $(1 - r) r < 1 - r < 1 - r^2$. 
So, $D_{21}(z)$ is decreasing (in $z>0$) and, obviously, $D_{21}(0+)=0$. Hence, $D_{21}(z)<0$ as well. 
Thus, by \eqref{eq:D2=}, $D_2(x)<0$, which shows that $D_1'(x)<0$ and $D_1(x)$ is decreasing -- in $x\in(0,1)$. 
Moreover, $D_1(0+)=r>0>r-1/r=D_1(1-)$. 
It follows, in view of \eqref{eq:D2}, that $D_1(x)$ changes in sign exactly once, from $+$ to $-$, as $x$ increases from $0$ to $1$. Equivalently, by \eqref{eq:D1}, $\pd{\ln R(p,b)}b$ changes in sign exactly once, from $+$ to $-$, as $b$ increases from $0$ to $1/2$. This completes the proof of Proposition~\ref{lem:}.  
\end{proof}

\begin{proof}[Proof of Theorem~\ref{prop:p}]\ 

\noindent\textbf{\eqref{ineq}}\quad 
To begin the proof of part \eqref{ineq} of Theorem~\ref{prop:p}, note that the last two inequalities in \eqref{eq:C_p} follow by the obvious symmetry 
\begin{equation}\label{eq:Rsymm}
R(p,b)=R(p,1-b) \quad\text{for all}\ b\in[0,1] 	
\end{equation}
and Proposition~\ref{lem:}. 
%part (ii) of Theorem~\ref{prop:p}, which has been already proved. 

Next, in view of the definition of $C_p$ in \eqref{eq:C_p}, 
inequality \eqref{eq:p} is a special case of \eqref{eq:centring}. 
Moreover, by the definition of $\trho$ in \eqref{eq:trho} and the homogeneity of the  
power function $|\cdot|^p$, 
\begin{equation}\label{eq:trho=}
\trho_{|\cdot|^p}(\la,b,t)=\rho_p(b,t):=\trho_{|\cdot|^p}(1,b,t)=\frac{\E|X_{1-b,b}|^p}{\E |X_{1-b,b}-t|^p}	
\end{equation}
for all $(\la,b)\in S_0$ and $t\in\R$, where $S_0$ is as in \eqref{eq:S_0}. 
Next, the denominator $\E |X_{1-b,b}-t|^p$ decreases in $t\in(-\infty,b-1]$, increases in $t\in[b,\infty)$, and attains its minimum over all $t\in[b-1,b]$ \big(and thus over all $t\in\R$\big) only at $t=t_b$, where $t_b$ is as in \eqref{eq:t_b}. 
%
%Substituting now $t_b$ for $t$ into the ratio in \eqref{eq:Cp=}, simplifying, and using the continuity of $R(p,b)$ in $b\in[0,1]$, 
%one has $C_p=\sup_{b\in(0,1/2]}R(p,b)=\max_{b\in[0,1]}R(p,b)$.  
%
So, 
\begin{equation}\label{eq:max trho=}	\max_{\la\in\R\setminus\{0\},\,t\in\R}\trho_{|\cdot|^p}(\la,b,t)=\max_{t\in\R}\rho_p(b,t)=\rho_p(b,t_b)=R(p,b)  
\end{equation}
for all $b\in(0,1/2]$, in view of \eqref{eq:R}. 
Now \eqref{eq:tc_f=c_f}, \eqref{eq:tc_f}, and \eqref{eq:Rsymm}  
yield  
\begin{equation*}
	c_{|\cdot|^p}=\sup_{b\in(0,1/2]}R(p,b)=\sup_{b\in[0,1]}R(p,b). %=\max_{b\in(0,1/2)}R(p,b)=R(p,b_p).   %=\max_{b\in[0,1]}R(p,b).  
\end{equation*} 
Thus, the proof of \eqref{eq:C_p} and all of part \eqref{ineq} of Theorem~\ref{prop:p} is complete. 

\noindent\textbf{\eqref{best}}\quad 
That the equality in \eqref{eq:p} obtains under either of the conditions (a) or (b) in part \eqref{best} of Theorem~\ref{prop:p} is trivial. 
If the condition (c) of part \eqref{best} holds with $\la=0$, then $X\eqD0$, and again the equality in \eqref{eq:p} is trivial. 
If now (c) holds with some $\la\in\R\setminus\{0\}$ -- so that $X\eqD\la(X_{1-b_p,b_p}-t_{b_p})$, then \eqref{eq:C_p}, \eqref{eq:max trho=}, and \eqref{eq:trho=} imply 
\begin{equation*}%\label{eq:C_p=rho_p(b,t_b)}
	C_p=R(p,b_p)=\rho_p(b_p,t_{b_p})=\frac{\E|X_{1-b_p,b_p}|^p}{\E |X_{1-b_p,b_p}-t_{b_p}|^p}
	=\frac{\E|X-\E X|^p}{\E |X|^p},  
\end{equation*} 
whence 
the equality in \eqref{eq:p} follows.  
Thus, for the equality in \eqref{eq:p} to hold it is sufficient that one of the conditions (a), (b), or (c) be satisfied. 

Let us now verify the necessity of one of these three conditions. W.l.o.g.\ condition (a) fails to hold, so that $\E|X|^p<\infty$. If now $p=2$ then $C_p=C_2=1$, and the necessity of the condition $\E X=0$ for the equality in \eqref{eq:p} is obvious. It remains to consider the case when $p\ne2$ and $\E|X|^p<\infty$. 
Suppose that one has the equality in \eqref{eq:p} and let $f=|\cdot|^p$. Then, by the definition of $C_p$ in \eqref{eq:C_p} and the equality \eqref{eq:tc_f=c_f}, equalities take place in \eqref{eq:le tc_f} and \eqref{eq:le int S}. %, with $f=|\cdot|^p$. 
In view of the condition $\E|X|^p<\infty$, 
the integrals in \eqref{eq:le tc_f} and \eqref{eq:le int S} are both finite and equal to each other. 
So,  
the equality in \eqref{eq:le int S} means that  
$|\E X|^p\,\mu\big(\{0\}\times(0,1/2]\big)=0$.  
If now $\mu\big(\{0\}\times(0,1/2]\big)\ne0$ then $\E X=0$, and the equality in \eqref{eq:p} takes the form $\E|X|^p=C_p\E|X|^p$; but, by part \eqref{C_p mono} of Theorem~\ref{prop:p} (to be proved a bit later), the condition $p\ne2$ implies $C_p>1$, which yields $\E|X|^p=0$, and so, $X\eqD\la(X_{1-b_p,b_p}-t_{b_p})$ for $\la=0$. 
It remains to consider the case when $p\ne2$, $\E|X|^p<\infty$, and $\mu\big(\{0\}\times(0,1/2]\big)=0$. Then $\mu(S_0)=\mu(S)=1$, and the equality in \eqref{eq:le tc_f} (again with $f=|\cdot|^p$), together with \eqref{eq:C_p} and \eqref{eq:tc_f=c_f}, will imply that 
	$\E|\la X_{1-b,b}|^p=C_p\E|\la X_{1-b,b}+\E X|^p$ for $\mu$-almost all $(\la,b)\in S_0$. 
In view of \eqref{eq:trho=}, \eqref{eq:C_p}, Proposition~\ref{lem:}, and \eqref{eq:max trho=}, 
this in turn yields 
\begin{equation*}%\label{eq:chain}
\rho_p(b,-\E X/\la)%=\rho_p(b,t_{b_p})
=R(p,b_p)\ge R(p,b)=\rho_p(b,t_b)  	
\end{equation*}
for $\mu$-almost all $(\la,b)\in S_0$. 
Now recall that for each $b\in(0,1/2]$ the maximum of $\rho_p(b,t)$ in $t\in\R$ is attained only at $t=t_b$. 
It follows that for $\mu$-almost all $(\la,b)\in S_0$ one has  
\begin{enumerate}[(i)]
	\item $R(p,b_p)=R(p,b)$ and hence, by Proposition~\ref{lem:}, $b=b_p$ and 
	\item   
$-\E X/\la=t_b=t_{b_p}$ 
%$-\E X/\la=t_b$ 
or, equivalently, $\la=-\E X/t_b=-\E X/t_{b_p}=:\la_p$. 
\end{enumerate} 
Therefore, $(\la,b)=(\la_p,b_p)$ for $\mu$-almost all $(\la,b)\in S_0$ and thus for $\mu$-almost all $(\la,b)\in S$. 
Now \eqref{eq:g} shows that $X+\la_p t_{b_p}=X-\E X\eqD\la_p X_{1-b_p,b_p}$ or, equivalently, $X\eqD\la_p(X_{1-b_p,b_p}-t_{b_p})$, which completes the proof of part \eqref{best} of Theorem~\ref{prop:p}. 

\noindent\textbf{\eqref{symm}}\quad  
Part \eqref{symm} of Theorem~\ref{prop:p} follows immediately by the symmetry \eqref{eq:R-symm} of $R(p,b)$ in $p$ and the definitions of $C_p$ and $b_p$ in \eqref{eq:C_p} and Proposition~\ref{lem:}, respectively. 

\noindent\textbf{\eqref{asymp}}\quad 
As in \eqref{eq:r,x,z}, let $r:=p-1$, so that $r\to\infty$. 
For a moment, take any $k\in(0,\infty)$ and choose $b=\frac kr$. 
Then, by \eqref{eq:r,x,z}, $x\sim b=\frac kr$, and now \eqref{eq:D1} yields  
$D_1(r,x)\sim(1-\frac1{2k})r$, whence $D_1(r,x)$ is eventually (i.e., for all large enough $r$) positive or negative according as $k$ is greater or less than $\frac12$. 
So, again by \eqref{eq:r,x,z}, for any real $\check k$ and $\hat k$ such that $0<\check k<\frac12<\hat k$, 
eventually $\pd{R(p,b)}b\big|_{b=\check k/r}<0<\pd{R(p,b)}b\big|_{b=\hat k/r}$.  
It follows by Proposition~\ref{lem:} that 
\begin{equation}\label{eq:b_p sim}
b_p\sim\frac1{2r}, 	
\end{equation}
that is, $b_p=\ka/r$ for some $\ka$ varying with $r$ so that $\ka\to1/2$. 
Hence, 
\begin{equation}\label{eq:factor1}
	(1-b_p)^r+b_p^r=(1-\ka/r)^r+(\ka/r)^r\to e^{-1/2}. 
\end{equation}
Next, $b_p^{1/r}=(\ka/r)^{1/r}=\exp\big(\frac1r\,\ln\frac\ka r\big)=1+\frac1r\,\ln\frac\ka r+O\big(\big(\frac1r\,\ln\frac\ka r\big)^2\big)$ and \break 
$(1-b_p)^{1/r}=1+O(1/r^2)$, whence 
%$(1-b_p)^{1/r}+b_p^{1/r}=2(1+\frac1{2r}\,\ln\frac\ka r+O\big(\frac{\ln^2 r}{r^2}\big)$  
\begin{align*}
	\big((1-b_p)^{1/r}+b_p^{1/r}\big)^r
	&=\Big[2\Big(1+\frac1{2r}\,\ln\frac\ka r+O\Big(\frac{\ln^2 r}{r^2}\Big)\Big)\Big]^r \\ 
	&=\Big[2\exp\Big\{\frac1{2r}\,\ln\frac\ka r+o\Big(\frac1r\Big)\Big\}\Big]^r
	\sim2^r\sqrt{\frac\ka r}\sim\frac{2^p}{\sqrt{8p}}. 
\end{align*}
Recalling now \eqref{eq:C_p}, \eqref{eq:R}, and \eqref{eq:factor1}, one obtains \eqref{eq:C_p sim}. 

\noindent\textbf{\eqref{C_p mono}}\quad 
Take any $b\in(0,1/2)$. Then 
%$b^r + (1 - b)^r$ is strictly log-convex in $r>0$, as the sum of strictly log-convex functions;  
%alternatively, 
$$d_{2,1}(r):=\pd{}r\pd{}r\ln\big(b^r + (1 - b)^r\big)=%\break 
\frac{(1-b)^r b^r}{\big(b^r + (1 - b)^r\big)^2}\, \ln^2\frac{1-b}b>0$$
for all $r>0$. 
Moreover, $d_{2,2}(r):=\pd{}r\pd{}r\ln\big[\big(b^{1/r} + (1 - b)^{1/r}\big)^r\big]=d_{2,1}(1/r)/r^3>0$ for all $r>0$. 
So, $\pd{}p\pd{}p\ln R(p,b)=d_{2,1}(p-1)+d_{2,2}(p-1)>0$, which shows that $R(p,b)$ is strictly log-convex in $p\in(1,\infty)$. 
Also, $\pd{}p\ln R(p,b)\big|_{p=2}=0$, so that $R(p,b)$ decreases in $p\in(1,2]$ and increases in $p\in[2,\infty)$, with $R(2,b)=1$.   
Therefore and in view of \eqref{eq:C_p} -- note in particular the attainment of the supremum there, $C_p$ is strictly log-convex and hence continuous in $p\in(1,\infty)$, and it also follows that $C_p$ decreases in $p\in(1,2]$ and increases in $p\in[2,\infty)$, with $C_p=1$. 
Next, \eqref{eq:C_p sim} shows that $C_p\to\infty$ as $p\to\infty$. 
Letting now $p\downarrow1$ and using \eqref{eq:symm}, one has $q\to\infty$ and hence 
$C_p=C_q^{1/(q-1)}=\big(2^q/\sqrt{(8+o(1))eq}\,\big)^{1/(q-1)}\to2$. 
This completes the proof of part \eqref{C_p mono} of Theorem~\ref{prop:p}. 

\noindent\textbf{\eqref{C_3}}\quad The proof of part \eqref{C_3} of Theorem~\ref{prop:p} is straightforward, in view of \eqref{eq:C_p}, Proposition~\ref{lem:}, \eqref{eq:R}, and \eqref{eq:t_b}.  
\end{proof}

\begin{proof}[Proof of Corollary~\ref{cor:concentr}]\ 
The proof is based on ideas presented in \cite{viniti10,pin-sakh} concerning the use of the mentioned Yurinski{\u\i} martingale decomposition; similar ideas were also used e.g.\ in \cite{bent-isr,normal,bahr-esseen}. 
Consider the martingale defined by the formula $\zeta_j:=\E_j(Y-\E Y)$ for $j\in\intr0n$, where $\E_j$ stands for the conditional expectation given the $\si$-algebra generated by $(X_1,\dots,X_j)$, with $\E_0:=\E$, and then consider the differences $\xi_i:=\zeta_i-\zeta_{i-1}$. 
Next, for each $i\in\intr1n$ introduce the r.v.\ 
\begin{equation*}
\eta_i :=\E_i(Y - \tilde Y_i), 	
\end{equation*}
where $\tilde Y_i := g(X_1,\dots,X_{i-1},x_i,X_{i+1},\dots,X_n)$, so that 
$\xi_i=\eta_i-\E_{i-1}\eta_i$, since the r.v.'s $X_1,\dots,X_n$ are independent.  
Also, in view of \eqref{eq:Lip} or \eqref{eq:LipE}, 
for all $i\in\intr1n$ and $z_i\in\X_i$ one has 
$|\eta_i|\le\rho_i(X_i,z_i)$, whence, by \eqref{eq:p}, 
\begin{align*}
\E_{i-1}|\xi_i|^r=\E_{i-1}|\eta_i-\E_{i-1}\eta_i|^r\le C_r\E_{i-1}|\eta_i|^r
&\le C_r\E_{i-1}\rho_i(X_i,z_i)^r \\ 
&=C_r\E\rho_i(X_i,z_i)^r 	
\end{align*}
for all $r\in(1,\infty)$.  
Now \eqref{eq:concentr} follows from \eqref{eq:mart}, since $\zeta_n=Y-\E Y$ and $C_2=1$. 
\end{proof}  

\bibliographystyle{abbrv}
%\bibliographystyle{ims}
%\bibliography{are.citations}
%\bibliography{citat}

%\bibliography{citations}

\bibliography{C:/Users/Iosif/Dropbox/mtu/bib_files/citations}
%\bibliography{C:/Users/Iosif/Documents/mtu_home01-30-10/bib_files/citations}
%\bibliography{C:/Users/Iosif/Documents/mtu_home12-22-08/bib_files/citations}

\end{document}